\documentclass[reqno,oneside,12pt,letterpaper]{amsart}

\usepackage[usenames,dvipsnames]{xcolor}
\usepackage[colorlinks=true,linkcolor=Blue,citecolor=NavyBlue]{hyperref}
\usepackage{amsmath,amssymb,epsf}
\usepackage{amsfonts,amsbsy,amscd,stmaryrd}
\usepackage{latexsym,amsopn}
\usepackage{amsthm,mathtools}
\usepackage{mathrsfs}\usepackage{esint}
\usepackage[mathscr]{eucal} 
\usepackage{autonum} 
\usepackage{comment}

\allowdisplaybreaks
 
\usepackage{graphicx}

\usepackage{color}
\definecolor{Red}{rgb}{1.,0.,0.}
\newcounter{smallarabics}
\newenvironment{arabicenumerate}
{\begin{list}{{\normalfont\textrm{(\arabic{smallarabics})}}}
  {\usecounter{smallarabics}\setlength{\itemindent}{0cm}
   \setlength{\leftmargin}{5ex}\setlength{\labelwidth}{4ex}
   \setlength{\topsep}{0.75\parsep}\setlength{\partopsep}{0ex}
   \setlength{\itemsep}{0ex}}}
{\end{list}}

\newcounter{smallroman}

\newcommand{\ben}{\begin{arabicenumerate}}  
\newcommand{\een}{\end{arabicenumerate}}

\newtheorem{theorem}{Theorem}[section]
\newtheorem{assumption}{Hypothesis}[section]
\newtheorem{proposition}[theorem]{Proposition}
\newtheorem{lemma}[theorem]{Lemma}

\theoremstyle{definition}
\newtheorem{definition}[theorem]{Definition}
\newtheorem{remark}[theorem]{Remark}
\newtheorem{example}[theorem]{Example}
\newcommand{\beq}{\begin{equation}}
\newcommand{\eeq}{\end{equation}}

\newcommand{\bea}{\begin{aligned}}
\newcommand{\eea}{\end{aligned}}

\newcommand{\bex}{\begin{example}}
\newcommand{\eex}{\end{example}}
\def\bel{\begin{lemma}}
\def\eel{\end{lemma}}
\def\bet{\begin{theoreme}}
\def\eet{\end{theoreme}}
\def\bed{\begin{definition}}
\def\eed{\end{definition}}
\def\ber{\begin{remark}}
\def\eer{\end{remark}}
\def\qeds{\qed}

\def\beproof{\noindent{\bf Proof.}\ }
\renewenvironment{proof}{\beproof}{\qed}

\renewcommand{\S}{$\mathsection$}
\renewcommand{\mod}{\,\mbox{ mod } \,}

\def\fantom{\\ &\phantom{=}\,}
\renewcommand{\leq}{\leqslant}
\renewcommand{\geq}{\geqslant}
\newcommand{\open}[1]{\mathopen{}\mathclose{\left]#1 \right[}}
\newcommand{\opens}[1]{\mathopen{}\mathclose{]#1 [}}
\newcommand{\clopen}[1]{\mathopen{}\mathclose{\left[#1 \right[}}
\newcommand{\opencl}[1]{\mathopen{}\mathclose{\left]#1 \right]}}
\newcommand{\norm}[1]{\left\|{#1}\right\|}
\newcommand{\module}[1]{\left|#1\right|}

\newcommand{\bra}{\langle}
\newcommand{\ket}{\rangle}
\def\st{{ \ |\  }}

\newcommand*{\defeq}{:=}

\def\rr{{\mathbb R}}

\def\cc{{\mathbb C}}

\def\cC{{{C}}}

\def\cD{\mathcal{D}}
\def\ham{\mathcal{X}}

\renewcommand{\i}{i}

\renewcommand{\Im}{\operatorname{Im}}
\renewcommand{\Re}{\operatorname{Re}}
\renewcommand{\sp}{\operatorname{sp}}

\DeclareMathOperator{\supp}{supp}

\DeclareMathOperator{\Ker}{Ker}
\DeclareMathOperator{\Ran}{Ran}
\DeclareMathOperator{\Dom}{Dom}
\DeclareMathOperator{\wf}{WF}

\newcommand{\one}{{I}}

\def\p{\partial}
\def\cf{\cC^\infty}

\def\zero{0}
\def\12{\frac{1}{2}}

\DeclareMathOperator{\elll}{ell}

\renewcommand{\hbar}{{}}

\newcommand{\ON}{\| u\|_{{-N}}}
\newcommand{\ONs}{ C \| u\|^2_{{-N}}}

\newcommand{\bTM}{\mkern1.5mu\overline{\mkern-1.5mu T}{}^* M}
\newcommand{\bSig}{\mkern0.5mu\overline{\mkern-0.5mu \Sigma}}

\newcommand{\pr}{{\rm pr}}

\renewcommand{\ham}{H}

\newcommand{\epsnu}{\nu}

\def\st{{ \ |\  }}

\begin{document}




\title[Long-time evolution of forced waves in the low viscosity regime]{Long-time evolution of forced waves \\ in the low viscosity regime}

\author{Nicolas Frantz}
\address{LAREMA, Université Angers, France}
\email{nicolas.frantz@univ-angers.fr}
\author{Micha{\l} \textsc{Wrochna}}
\address{Mathematical Institute, Universiteit Utrecht, The Netherlands \vspace{-0.3cm}} \address{Mathematics \& Data Science, Vrije Universiteit Brussel, Belgium}
 \email{{m.wrochna@uu.nl}}
\email{michal.wrochna@vub.be}

\begin{abstract} We consider a model for internal waves described by a zero order pseudo-differential Hamiltonian $P$ damped by a second order viscosity term $i \epsnu Q$. Under Morse--Smale or similar weaker global conditions on the classical dynamics, we describe qualitatively the long-time behavior of solutions of the corresponding  evolution equation with smooth forcing  in a small $\epsnu$ regime. We show that dissipation effects arise no earlier than at the $t\sim \epsnu^{-1/3-}$ time scale.
\end{abstract}

\maketitle

\section{Introduction and main result}

\subsection{Introduction} Colin de Verdière and Saint-Raymond  introduced in \cite{CS19} a model for internal waves in fluids in the presence of topography, governed by a $0$th order pseudodifferential operator $P=P^*$ with Morse--Smale dynamics on a closed surface $M$. As shown in \cite{CS19} using  Mourre theory methods, the model captures the formation of singular profiles (or \emph{attractors}) as $t\to +\infty$ for solutions of the equation with given smooth periodic forcing
\beq\label{eq:sve}
(i \p_t - P) u_0(t)=  f e^{-i \omega_0 t}.
\eeq
The subsequent work \cite{CdV} generalized the result to arbitrary dimension and weaker dynamical assumptions, and
  an alternative microlocal approach based on radial estimates  was proposed by Dyatlov--Zworski \cite{DZ} who also uncovered the role played by Lagrangian regularity.

A significant  drawback of this model is that it does not take viscosity into account. A more realistic version consists in adding an elliptic second order operator $Q$, typically a Laplace--Beltrami operator $-\Delta$, or $-\Delta+\one$, and considering the small $\epsnu$ behavior of  solutions of the viscous equation
\beq\label{eq:ve}
(i \p_t - P+i \epsnu Q) u_\epsnu(t)=  f e^{-i \omega_0 t}.
\eeq
The primary difficulty is that $Q$ is \emph{two orders greater} than $P$ so the spectral theory of the elliptic operator $P_\epsnu:=P-i \epsnu Q$ is vastly different from that of $P$, and the role of eigenvalues of $P_\epsnu$ in the description of the $\epsnu\to 0+$, $t\to +\infty$ behavior of \eqref{eq:ve} is unclear.

Insights into the relationship of  $P_\epsnu$ eigenvalues with resonances of $P$ are provided  by results of Galkowski--Zworski  \cite{Galkowski2022} in the $Q=-\Delta$ case, who showed convergence of eigenvalues of $P_\epsnu$ close to $0$ to resonances,  and  of Wang \cite{Wang2022} who proved linear convergence rate a generic absence of embedded eigenvalues result; see also \cite{Almonacid2024} for a numerical study. 

In the present paper we focus instead on the large $t$ behavior in the low viscosity regime. The interplay between the two limits is expected to lead to different behavior of the system depending on the relative scale, and our main  objective is to identify a regime in which viscosity contributes significantly but the induced damping effects are not overwhelming.

\subsection{Setting  and main result} Before formulating our results let us introduce the notation and main assumptions. Note that without loss of generality we can assume that $\omega_0=0$.

Let $M$ be a compact manifold (without boundary) of dimension $n\geq 2$. We denote by $S^m(T^*M)$ the usual symbol class and by $S^m_{\rm h}(T^*M)$ the subclass of \emph{homogeneous} ones. We use the standard notation $\Psi^m(M)$ for pseudo-differential operators of order $m\in\rr$, see e.g.~\cite{DZbook} for a textbook introduction.

If $d\mu$ is a smooth density on $M$, we denote by $\bra \cdot,\cdot\ket$ the $L^2(M,d\mu)$ scalar product and $\| \cdot\|$ the norm.

We consider a pseudo-differential operator $P\in\Psi^0(M)$ and following \cite{CdV,DZ} we consider the following setup.

\begin{assumption}\label{hyp:main} We assume:
\ben
\item $P^*=P$ in $L^2(M,d\mu)$ for some smooth density $d\mu$;
\item the principal symbol of $P$, denoted in what follows by $p$, belongs to $S^0_{\rm h}(T^*M)$;
\item\label{a:regular} $0$ is a regular value of $p$, i.e.~$dp\neq 0$ on $p^{-1}(0)$.
\een
\end{assumption}

Assumption \eqref{a:regular} ensures that  $\Sigma_\omega:=p^{-1}(\omega)$ is a smooth conic submanifold of $T^*M\setminus \zero$ for all $|\omega|\leq \delta$ with $\delta >0$ small enough. We then make the same global non-trapping  assumption on  the Hamilton flow of $p$ as in \cite{CdV}, formulated in terms of weakly hyperbolic \emph{attractors} and \emph{repulsors} at infinite frequencies (see Definition \ref{defat}).

\begin{assumption}\label{hyp:main2} We assume that $P$ has \emph{simple structure}, i.e.~there are weakly hyperbolic attractors/repulsors $L_0^\pm$  such that forward Hamilton trajectories in $\Sigma_0$ tend to $L_0^+$ and backward trajectories tend to  $L_0^-$ (in the sense of Definitions \ref{def:basin} and \ref{def:simple}).
\end{assumption}

As discussed in \cite{CdV}, a special case is the Morse--Smale setting considered in \cite{CS19} and studied at length in \cite{DZ}. The general assumption is also  closely related to the setting of  sources and sinks which suffices for the estimates in the work of Dyatlov--Zworski \cite{DZ}; see Remark \ref{rmk:ss}.

The viscosity is modelled by an invertible operator
\beq\label{eq:Q}
Q\in\Psi^2(M) \mbox{ elliptic, s.t. } Q\geq 0.
\eeq
We abbreviate $L^2(M)=L^2(M,d\mu)$ and for the sake of simplicity we use the norm $\| u \|_s= \|Q^{s/2} u \|$ on the Sobolev space $H^s(M)$. Throughout the paper we write $u\in H^{s-}(M)$ to mean $u\in H^{s-\alpha}(M)$ for small enough $\alpha>0$ (where $\alpha$ can vary from line to line) and the same principle applies to $\epsnu^{-1/3-}$ and similar notation.

Our main result can be summarized as follows. 

\begin{theorem}[{cf.~Theorem \ref{thm:main} and Proposition \ref{prop:main}}]
Assume Hypotheses \ref{hyp:main}--\ref{hyp:main2} and $0\notin \sp_{\rm pp}(P)$. Then for any $f\in C^\infty(M)$, the solution  of  \eqref{eq:ve} with $u_\epsnu(0)=0$ decomposes as  \beq\label{eq:decomp0}
	u_\epsnu(t)=u_{\epsnu,\infty}+b_\epsnu(t)+e_{\epsnu}(t),	
	\eeq
	where $u_{\epsnu,\infty}= -P_\epsnu^{-1} f\to -(P-i 0)^{-1}f$ in $H^{-\frac{1}{2}-}(M)$ as $\epsnu\to 0+$, $\| b_\epsnu(t)\|\leq C\|f\|_1$ uniformly in   $t> 0$, $\epsnu> 0$, and for all $\delta_1>0$ there exists $\delta_2>0$ such that
\beq\label{eq:decomp2}
 \quad \|e_{\epsnu}(t)\|_{-1/2-}\leq  C t^{-\delta_2}  \| f\|, 
\eeq
	uniformly for  $t \sim \nu^{-\frac{1}{3}-\delta_1}$.  Furthermore,   if in addition $f\in \Ran \one_{[-\delta,\delta]}(P)$ for $\delta>0$ small enough, then for $t\in \opens{\epsnu^{-\frac{1}{3}-},\infty}$,  $u_\epsnu(t)\to  u_0(t)$ uniformly in $H^{-\12-}(M)$ as $\epsnu\to 0+$.
\end{theorem}

Thus  the dissipation effects have to arise at the  $t\sim \epsnu^{-1/3-}$ scale  as predicted by various heuristics\footnote{We thank Yves Colin de Verdière, Charlotte Dietze, Laure Saint-Raymond and Thierry Gallay for sharing their insights on this problem.}.

\subsection{Bibliographical remarks} We mention the  physics literature only very briefly, featuring in particular the work of Maas et al.~\cite{maas}; the importance of taking into account viscosity is stressed  by \cite{ogilvie,Rieutord2010}.  

The spectral and scattering theory of $P$ was studied recently in various settings closely related to ours; see \cite{wangreview} for a concise review. On top of the already mentioned results from \cite{CS19,DZ,CdV,Galkowski2022,Wang2022},  Wang  \cite{Wang2023}  showed that the scattering matrix of $P$ is a Fourier Integral Operator. Christianson--Wang--Wang \cite{Christianson2024} introduced control estimates and proved the disappearance of singular patterns in the presence of damping. Spectral theory of $0$th order operators on $\mathbb{S}^1$ was considered by Tao \cite{Tao2024} who gave in particular an example of embedded eigenvalue.

We note that the assumption that $M$ is compact without boundary is a significant simplification; more realistic  models are studied by Dyatlov--Li--Wang, \cite{Dyatlov2025}, Li \cite{Li2023,Li2024} and Li--Wang--Wunsch \cite{Li2024a}. Naturally, this prompts the question of whether our results can be generalized to  settings with boundary.

\subsection{Structure of proofs}

The main idea here is to combine G\'erard's approach to Mourre theory \cite{gerard} with microlocal estimates near radial sets in the spirit of the works of Melrose \cite{melrose}, Vasy \cite{Vasy2013} and Dyatlov--Zworski \cite{DZ16,DZbook}. This allows us to prove radial estimates for $P-i\epsnu Q$ directly and deal with the viscosity term by an iterated correction of the corresponding term in the positive commutator estimates. This is the content of Section \ref{s:radial} where in particular we deduce a uniform bound $(P-\omega - i \epsnu Q)^{-1}=O(\epsnu^{-1/3})$  in  $B(H^{-\12-}(M))$ for small  $\module{\omega}$, suitable for composition purposes later on. 

In Section \ref{s:spectral} we discuss general properties of $(P_\epsnu-\omega)^{-1}=(P-\omega - i \epsnu Q)^{-1}$ that do not make use of dynamical assumptions. Ideally we would like to have a good control of $(P_\epsnu-\omega)^{-N}\psi(P)$ as $\epsnu\to 0+$ for a suitable spectral cutoff $\psi$. Possible $H^1(M)$ eigenvalues of $P$ prevent us from getting that directly for large $N$, but we show a related result which is equally useful in the context of contour integrals. 

These results are combined in Section \ref{s:final} to prove the main theorem  by expressing $u_\epsnu(t)$ in terms of powers of the  resolvent (through a contour integration formula for the semi-group). At this point the main difficulty is that there is no obvious  way of  introducing spectral cutoffs consistently for $\epsnu>0$ and $\epsnu=0$, so the  decomposition requires special care. 

The necessary background and preliminary results on pseudo-differential calculus are briefly introduced in Appendix \ref{app:A}.

\section{Radial estimates and zero viscosity limit} \label{s:radial}

\subsection{Microlocal positive commutator estimates}  

If $A\in \Psi^m(M)$ we denote by $\sigma_\pr(A)$ its principal symbol. The microsupport of $A$ (or primed wavefront set in the sense of pseudo-differential calculus) is denoted by $\wf'(A)$, and the elliptic set by $\elll(A)$. Recall that $\wf'(A)$ is closed and $\elll(A)$ is open. We will often use well-known variants of the elliptic estimate and of the sharp G\r{a}rding inequality, which are recalled in Appendix \ref{ss:basic}. 

In this section we generalize the setting slightly by allowing the viscosity term to be of arbitrary order $\ell\geq 0$ (this allows in particular to include the case $Q=\one$ for the sake of comparison), thus \eqref{eq:Q} is replaced by
\beq\label{eq:Q2}
Q\in\Psi^\ell(M) \mbox{ is elliptic, s.t. } Q> 0.
\eeq
Note that this implies that the operator $P-\i \epsnu Q$ is elliptic for $\epsnu\neq 0$.

We start with a lemma that summarizes the positive commutator method in a pseudo-differential setting, where the complex absorption term $Q$ is not assumed to be of lower order nor to have special commutativity properties. This motivates a careful preparation of the commutant in step \textbf{1.}~of the proof.

\begin{lemma}\label{lem:main} Assume Hypothesis \ref{hyp:main}, and let $Q$ be as  in \eqref{eq:Q}. Let $B\in\Psi^0(M)$. Suppose that there exists $m,s\in\rr$, $G_1\in\Psi^{m-s}(M)$ and $G_2\in\Psi^{s}(M)$ such that: 
\beq\label{cond1}
 \pm G_1 G_2\geq 0 \mbox{ on } \cf(M),
\eeq
$[P, i G_1 G_2]\in\Psi^{2s}(M)$, and
\beq\label{cond2}
\sigma_\pr\big([P, \i G_1 G_2] -G_2^* G_2\big)\geq 0 \mbox{ on } T^*M\setminus\elll(B).
\eeq
Let $B_1\in\Psi^0(M)$ be such that $\wf'(G_2)\subset\elll(B_1)$. Then for all $N$\footnote{We omit the dependence on $N$ and other Sobolev orders of the positive constants $C$.} and $u\in\cf(M)$,
\beq\label{eq:lem1}
\|  G_2 u \| \leq  C( \|B u\|_{{s}}  +   \| (P-\omega\pm \i\epsnu Q)u \|_{{m-s}}  +  \| B_1  u \|_{{s-1/2}} + \ON)
\eeq
uniformly in $\epsnu\geq 0$ and $\omega\in\rr$.
\end{lemma}
\proof \textbf{1.} In the first step we will construct an operator $G\in\Psi^m(M)$ with the same principal symbol as $G_1 G_2$, but with different positivity properties. We start by defining
\[
\bea
\Gamma : \Psi(M)&\to \Psi(M) \\
   A &\mapsto  \Re(Q^\12 A Q^{-\12}),
\eea
\]
where we use the notation $\Re A\defeq (A+A^*)/2$. Then ${\one}-\Gamma: \Re\Psi^{p}(M)\to\Re\Psi^{p-1}(M)$ for all $p\in\rr$. We define $G\in\Psi^{m}(M)$ by the asymptotic sum
\[
G \sim \sum_{j=0}^{\infty} ({\one}-\Gamma)^j(G_1 G_2).
\]
Then $\sigma_{\rm pr}(G)=\sigma_{\rm pr}(G_1 G_2)$, $\wf'(G)\subset\wf'(G_1 G_2)$ and $G^*=G$. Furthermore, 
\[
\bea
\pm\Re GQ &= \pm Q^{\12} (\Re  Q^{-\12} G Q^{\12}) Q^{\12}=\pm Q^{\12} \Gamma(G) Q^{\12}\\
&=\pm Q^{\12}\big( G - ({\one}-\Gamma)(G)\big) Q^{\12} =  \pm Q^{\12} G_1 G_2 Q^{\12} \geq 0 \mod \Psi^{-\infty}(M),
\eea
\]
where in the last step we used \eqref{cond1}. Thus,
\beq\label{eq:posGQ}
\mp \Re \bra  G Q u,u\ket \leq  \ONs.
\eeq
\textbf{2.} Next, by \eqref{cond2} we have
\[
\sigma_\pr\big([P, \i G] -G_2^* G_2\big)\geq 0 \mbox{ on } T^*M\setminus\elll(B).
\]
We can apply the microlocalized sharp G\r{a}rding inequality (recalled in Proposition \ref{garding}) to the operators $[P, \i G]-G_2^* G_2$, $B$ and $B_1$. This yields:
\beq\label{eq:fromgarding}
\| G_2  u\|^2 \leq \bra  [P, \i G] u,u\ket + C \|B u\|^2_{{s}} + C \| B_1 u \|^2_{{s-\12}} + C \| u\|^2_{{-N}}. 
\eeq

\textbf{3.} We now undo the commutator:
\[
\bea
\frac{1}{2}\bra  [P, \i G]  u,  u \ket &= \frac{\bra G (P-\omega) u, u \ket -\bra (P-\omega) G u, u \ket}{2 i}\\
&=\frac{\bra   (P-\omega) u,  G u \ket - \bra  G u, (P-\omega)u \ket}{2 i}\\
&=\Im\bra (P-\omega)u,G u\ket = \Im\bra (P-\omega\pm i\epsnu Q)u,G u\ket \mp  \epsnu \Re \bra Q u,G u\ket \\
&=\Im\bra f_{\pm\epsnu},G u\ket \mp  \epsnu \Re \bra G Q u, u\ket,
\eea
\]
where we have denoted $f_{\pm\epsnu}= (P-\omega\pm i\epsnu Q)u$. By \eqref{eq:posGQ}, this implies that uniformly in $\epsnu\geq 0$ and $\omega\in\rr$, 
\beq\label{eq:tempZ}
\bra [P, \i G]  u,  u \ket\leq  C |\bra  f_{\pm\epsnu} ,G u\ket| + \ONs.
\eeq
Recall from step \textbf{1.}~that $G=G_1 G_2 + R$ for some $R\in\Psi^{m-1}(M)$. Using first the Cauchy--Schwarz inequality, we get for all $\epsilon>0$ 
\[
\bea
|\bra  f_{\pm\epsnu} ,G u\ket| &\leq \| f_{\pm\epsnu} \|_{{m-s}} \| G u \|_{{s-m}} \\
& \leq  \epsilon^{-1} \| f_{\pm\epsnu} \|_{{m-s}}^2 + \epsilon \| G u \|_{{s-m}}^2 \\
& \leq  \epsilon^{-1} \| f_{\pm\epsnu} \|_{{m-s}}^2 + \epsilon \| G_1 G_2 u\|_{{s-m}}^2 +  \epsilon \| R u \|^2_{{s-m}}\\
& \leq  \epsilon^{-1} \| f_{\pm\epsnu} \|_{{m-s}}^2 + C \epsilon \| G_2 u \|^2 +   C \epsilon \| B_1 u \|^2_{{s-1/2}}+\ONs,
\eea
\]
where in the last step we used that $\wf'(R)\subset \elll(B_1)$ and applied the elliptic estimate (Theorem \ref{elliptic}). Combining this with \eqref{eq:fromgarding} and  \eqref{eq:tempZ},  by fixing $\epsilon$ sufficiently small, relabelling the constants, and then estimating the square root we obtain \eqref{eq:lem1}.\qed

We now state a variant of Lemma \ref{lem:main} where we get better regularity at the cost of worse  behaviour in $\epsnu$.

\begin{lemma}\label{lem:main2} With the assumptions Lemma \ref{lem:main}, for any $r\in \open{s, \frac{m+\ell}{2}}$ and $A\in\Psi^0(M)$ such that $\wf'(A)\subset \elll (G_1 G_2)$ we have
\beq\label{eq:lem1v}
\|  A u \|_{r} \leq  C_\epsnu ( \|B u\|_{{s}}  +   \| (P-\omega\pm \i\epsnu Q)u \|_{{m-s}}  +  \| B_1  u \|_{{s-1/2}} + \ON),
\eeq
where  $
C_\epsnu = C \epsnu^{-(s-r)/(2s-m-\ell)}$.
\end{lemma}
\proof  We repeat the proof of  Lemma \ref{lem:main}, but this time we keep the term $\mp  \epsnu \Re \bra G Q u, u\ket$. In this way, we obtain the estimate  
\beq\label{var1}
\| G_2  u\| \pm \epsnu^{\12}  \Re \bra G Q u, u\ket \leq  C( \|B u\|_{{s}}  +   \| (P-\omega\pm \i\epsnu Q)u \|_{{m-s}}  +  \| B_1  u \|_{{s-1/2}} + \ON)
\eeq 
instead of  \eqref{eq:lem1}. On the other hand by $\eqref{eq:posGQ}$  and the standard approximate square root argument applied to  $G_1G_2$,
$$
\pm \Re \bra G Q u, u\ket = \| A_0 u \|_{(m+\ell)/2} - C \ON
$$ 
 for some $A_0\in\Psi^0(M)$ with $\elll(A_0)=\elll(G)$. Therefore, by using the elliptic estimate (Theorem \ref{elliptic}) twice, for any $A\in\Psi^0(M)$ with $\wf'(A)\subset \elll(G_2)$ (which implies  $\wf'(A)\subset \elll(A_0)$)   we obtain
 \beq\label{var2}
\|A u \|_s + \epsnu^{\12}  \| A u \|_{(m+\ell)/2}  \leq  \| G_2  u\| \pm \epsnu^{\12} \Re \bra G Q u, u\ket + C \ON.
\eeq
By interpolation in Sobolev spaces,
\beq\label{var3}
\|A u \|_r   \leq   C  \epsnu^{-(s-r)/(2s-m-\ell)} ( \|A u \|_s + \epsnu^{\12}  \| A u \|_{(m+\ell)/2}).
\eeq
The assertion follows by combining \eqref{var1}, \eqref{var2} and \eqref{var3}.
\qeds

\subsection{Propagation and radial estimates}

The Hamiltonian vector field of $p$ is denoted by $H_p$, and its flow by $\Phi_t$. 

As a first consequence of Lemma \ref{lem:main}, we can show that the Duistermaat--H\"ormander propagation of singularities theorem (presented for the sake of simplicity as an estimate for $u\in \cf(M)$) holds true in our setting.

\begin{proposition}\label{propagation} Let $s\in \rr$. If $A,B,\widetilde{B}\in \Psi^0(M)$ and for each $q\in \wf'(A)$ there exists $T\geq 0$ such that
\[
\Phi_{-T}(q)\in \elll(B), \mbox{ and } \Phi_{-t}(q)\in \elll(\widetilde{B}) \mbox{ for all  } t\in [0,T],
\]
then for all $N$ and $u\in\cf(M)$,
\beq\label{eq:propagation}
\|  A u \|_{s} \leq C( \|B u\|_{{s}}  +   \| \widetilde{B} (P-\omega- \i\epsnu Q)u  \|_{{s+1}}  + \ON)
\eeq
uniformly in $\epsnu\geq 0$ and $|\omega|\leq \delta$. 
\end{proposition}
\proof The proof is only a slight modification of the well-known one presented in \cite[Thm.~E.47]{DZbook}, so we only sketch it. If $G,Y$ are the operators defined in (E.4.16) in \cite{DZbook}, then we can apply Lemma \ref{lem:main} with $G_1 G_2=-G^* G$ and  $G_2$ proportional to $Y G$. It then remains to repeat the steps following (E.4.27) in \cite{DZbook}.  \qed 

\medskip

Let $\zero$ be the zero section of $T^*M$. We denote by $\bTM$ the fiber radial compactification of the cotangent bundle (see e.g.~\cite[E.1.3]{DZbook}). Its boundary $\partial\bTM$ is diffeomorphic to the quotient of $T^*M\setminus\zero$ by dilations in the dual variables. The quotient map is denoted by
\[
\kappa: \bTM\setminus\zero \to \p\bTM.
\]

Note that $p$ extends to a smooth function on $\bTM$. The closure of the characteristic set $\Sigma_\omega= p^{-1}(\omega)$ of $P-\omega$ in $\p\bTM$ is denoted by $\bSig_\omega$, and we set $\p \bSig_\omega=\p\bTM\cap \bSig_\omega$. The rescaled Hamiltonian vector field $|\xi| H_p$ commutes with dilations in $\xi$ and is homogeneous of order $0$, so it defines a vector field
$$
X\defeq \kappa_*(|\xi| H_p) 
$$
on $\bSig_\omega$. Its flow is also denoted by  $\Phi_t$.

We say that a set $L\subset \bTM$ is $\Phi$-invariant if $\Phi_t(L)\subset L$ for all $t\in\rr$.

The analysis in \cite{CdV} motivates the following definition.
\begin{definition}\label{defat} We say that a $\Phi$-invariant closed set $L^\pm_\omega\subset \p\bSig_\omega$ is a \emph{weakly hyperbolic attractor}/\emph{repulsor} if:
\ben
\item there exists an open neighborhood $U$ of $L^\pm_\omega$ in $\p\bSig_\omega$ such that $\bigcap_{\pm t \geq 0}\Phi_t(U)=L^\pm_\omega$;
\item there exists $\beta>0$ and $k\in S^1_{\rm h}(T^*M)$  such that:
\begin{enumerate}
\item[a)]\label{hshhw1}  $\pm k>0$ on $\Lambda^\pm_\omega$, 
\item[b)]\label{hshhw2} $\ham_p  k  >  \beta$ on $\Lambda^\pm_\omega$,
\end{enumerate}
\een  
where we set $\Lambda^\pm_\omega =\kappa^{-1}(L^\pm_\omega)\subset \bTM\setminus\zero $.
\end{definition}

\begin{remark}\label{rmk:ss} It follows from \cite[Prop.~3.2]{CdV} that any weakly hyperbolic repulsor $L^-_\omega$ is a \emph{radial source} (in the precise sense of \cite[Def.~E.50]{DZbook}) for the first-order symbol $-k (p-\omega)$, and similarly $L^+_\omega$ is a \emph{radial sink}, cf.~\cite[Lem.~2.1]{DZ} for a very closely related statement in the Morse--Smale case, with $-k$ replaced by $|\xi|$ . This allows one to use radial estimates as in \cite{DZ,CdV} after a suitable generalization to account for the presence of the viscosity term $Q$ using Lemma \ref{lem:main}.  We proceed here however slightly differently and derive estimates for the zero-order operator $P$ directly.
\end{remark}

Since we are interested in a neighborhood of $\omega_0=0$, in this subsection we assume $L_\omega^\pm\subset  \p\bSig_\omega$ are weakly hyperbolic attractors/repulsors for $\omega\in[-\delta,\delta]$, and we set
\beq\label{eqlpm}
L^\pm:= \bigcup_{\omega\in [-\delta,\delta]}L_\omega^\pm, \quad  \Lambda^\pm:=  \bigcup_{\omega\in [-\delta,\delta]}\Lambda_\omega^\pm.
\eeq

\begin{definition}\label{def:basin} The \emph{forward}/\emph{backward} \emph{basin} of a $\Phi$-invariant set $L\subset \bigcup_{\omega\in [-\delta,\delta]} \bSig_\omega$, denoted by $\Phi^\pm(L)$,  is the set of all $q\in p^{-1}([-\delta,\delta])$ such that $\Phi_t(q)\to L$ as $t\to\pm \infty$.  
\end{definition}

Note in particular the inclusion $\Lambda^\pm\subset \Phi^\pm(L^\pm)$; note also that $L^\pm$ are {forward}/{backward} {basins}. 

One can find an ``escape function'' $k$ with the following better properties. 

\begin{lemma}\label{lembetterk} For any $q\in \Phi^\pm(L^\pm)$ there exists $\beta>0$ and $k\in S^1_{\rm h}(T^*M)$ such that $\supp k\cap  p^{-1}([-\delta,\delta])  \subset  \Phi^\pm(L^\pm)$ and:
\begin{enumerate}
\item[a)]\label{hshhww1}  $\pm k>0$ on a conic neighborhood of $\Lambda^\pm$ containing $q$,
\item[b)]\label{hshhww2} $\ham_p  k  >  \beta$ on $\Phi^\pm(L^\pm)$.
\end{enumerate}
\end{lemma}

\proof Let $k_1$ be a symbol satisfying a) and b) of Definition \ref{defat}. We can assume that $\supp k_1\cap p^{-1}([-\delta,\delta]) \subset  \Phi^\pm(L^\pm)$. Next, we proceed exactly as in \cite[Sec.~3.2.2]{CdV}. Namely, in the attractor case (and similarly in the repulsor case) we take $k_2\in S^1_{\rm h}(T^*M)$ such that
\[
k_2 =\lim_{t\to\infty}\left(k_1\circ \Phi_t-\textstyle\int_0^t m\circ \Phi_s \right) \mbox{ on } \Phi^+(L^+)
\]
 for some positive $m\in S^0_{\rm h}(T^*M)$ that equals $\ham_p k$ on a $\Phi$-invariant conic neighborhood $U$ of $\Lambda^\pm$. This way we obtain a symbol $k_2$ satisfying all the requested properties except that we do not have necessarily $\pm k_2(q)>0$. However, since $q\in \Phi^\pm(L^\pm)$, we have $\Phi_{s}(q)\subset U$ for some $s\in\rr$. Thus, the symbol $k\defeq k_2\circ \Phi_s$ has all the stated properties.\qed  

\medskip

The purpose of Lemma \ref{lem:B} below (which plays an analogous role to \cite[Lem.~E.53]{DZbook}) is to have a symbol $c\in S^0(T^*M)$ that will serve to microlocalize around $q$ without losing control of positivity of Poisson brackets whenever possible.

\begin{lemma}\label{lem:B}  Let $q$ and $k$ be as in Lemma \ref{lembetterk}. Then there exists $c\in S^0(T^*M)$ such that:
\begin{enumerate}
\item[a)]\label{wqewegt0} $c\geq 0$ everywhere, $c>0$ at $q$,
\item[b)]\label{wqewegt1} $c=0$ in a neighborhood of $p^{-1}([-\delta,\delta])\setminus \Phi^\pm(L^\pm)$ containing $\{k=0\}$,
\item[c)]\label{wqewegt2} $\pm\ham_p c\geq 0$ on $\Phi^\pm(L^\pm)$.
\end{enumerate} 
\end{lemma}
\proof Set $\alpha\defeq\sup\left|\ham_p \bra\xi \ket \right|$. For $\epsilon>0$, let $\chi_{\epsilon}\in C^\infty(\rr;[0,1])$ be such that 
\[
\chi_{\epsilon}\equiv 0 \mbox{ on } ]-\infty,\textstyle\frac{\epsilon}{4\alpha}], \ \ \chi_{\epsilon}\equiv 1 \mbox{ on } [\textstyle\frac{\epsilon}{2\alpha},+\infty[,
\]
 and $\chi'_{\epsilon}\geq 0$ everywhere. 

 In the attractor case, let $c\defeq \chi_{\epsilon} (\bra\xi \ket^{-1}k)$, with $\epsilon\leq \beta/2$ small enough to ensure $c>0$ at $q$. We have
\[
\ham_p c = \chi'_{\epsilon}(\bra\xi \ket^{-1} k)\frac{\bra\xi\ket\ham_p k -k \ham_p\bra\xi\ket}{\bra\xi\ket^2}\geq \chi'_{\epsilon}(\bra\xi \ket^{-1} k)\frac{\bra\xi\ket \beta -k \alpha}{\bra\xi\ket^2} \geq 0
\]
on $\Phi^+(L^+)$ since $\frac{\epsilon}{4\alpha} \bra \xi\ket \leq k\leq \frac{\epsilon}{2\alpha} \bra \xi\ket$ on $\supp\chi'_{\epsilon}$. Besides, $c=0$ on $\{k\leq \frac{\epsilon}{4\alpha} \}$. Thus, the second part of b) follows from $\supp k\cap p^{-1}([-\delta,\delta]) \subset  \Phi^\pm(L^\pm)$.

In the negative case the proof is analogous with $c\defeq \chi (-\bra\xi \ket^{-1}k)$. \qed

We prove below a radial estimate which gives regularity in the basin of a repulsor.

\begin{proposition}\label{lem:consp} Let $s\in \rr$, $s\neq -\12$, and let $\delta>0$ be sufficiently small.  If $A\in\Psi^0(M)$ satisfies $\wf'(A)\cap p^{-1}([-\delta,\delta]) \subset \Phi^-(L^-)$, then for all $N$ and $u\in\cf(M)$,
\beq\label{eq:lemconsp}
\|  A u \|_{s} \leq  C ( \| (P-\omega-\i \epsnu Q) u \|_{{m-s}}   + \ON )
\eeq
uniformly in $\epsnu\geq 0$ and $|\omega|\leq\delta$, where $m=0$ if $s<-\12$ and $m=1+2s$ if $s>-\12$. 
\end{proposition}
\proof   \textbf{1.} By the elliptic estimate outside of $p^{-1}([-\delta,\delta])$ and a microlocal partition of unity argument it suffices to prove \eqref{eq:lemconsp} for any $A\in\Psi^0(M)$ such that $\wf'(A)$ is contained in a small neighborhood of an arbitrary point $q\in \Phi^-(L^-)$.    

Let us fix $q\in \Phi^-(L^-)$ and let $k\in S^1(T^*M)$ and $c\in S^0(T^*M)$ be as in Lemmas \ref{lembetterk}, \ref{lem:B}. In particular, $k<0$ on $T^*M\setminus \{c=0\}$. Let $K\in\Psi^1(M)$ be an elliptic quantization of a symbol that equals $k$ on a neighborhood of $T^*M\setminus \{c=0\}$ and such that $K \leq 0$. 

Our first objective is to show that the assumptions of Lemma \ref{lem:main} are satisfied. We will construct $G_1$ and $G_2$ as microlocalizations of suitably designed functions of $K$. In the two respective cases $s<-\12$ and $s>-\12$ we define a function $g$ as follows:

a) Assume $s<-\12$. Set 
\beq\label{eq:firstg}
g(\lambda)=-\int_{\lambda}^{\infty} \bra \tau \ket^{2s} d\tau,
\eeq
Then $g\in  S^0(\rr)\cap S^{1+2s}(\rr_+)$, $g < 0$. Furthermore, $g'(\lambda)=\bra \lambda \ket^{2s}$, so $g'\in S^{2s}(\rr)$ and $g'>0$.

b) Assume $s>-\12$. Note that in that case \eqref{eq:firstg} is ill-defined. Instead, let $\chi_-\in C^\infty(\rr;[0,1])$ be such that $\chi_-\equiv 1$ on $\opencl{-\infty,0}$ and $\chi_-\equiv 0$ on $\clopen{1,\infty}$, and set 
\[
g(\lambda)=-\int_{\lambda}^{\infty}\chi^2_-(\tau) \bra \tau\ket^{2s} d\tau,
\]
so that
\[
g'(\lambda)=\chi^2_-(\lambda) \bra \lambda\ket^{2s}.
\]
Then $g\in S^{1+2s}(\rr)\cap S^{-\infty}(\rr_+)$, $g\leq 0$, $g'\in S^{2s}(\rr)\cap S^{-\infty}(\rr_+)$ and $g'\geq 0$. Note that since $K\leq 0$, we have
\beq\label{eq:specialK}
g'(K)=\bra K\ket^{2s}. 
\eeq

Let $B_2\in\Psi^0(M)$ be the quantization of $c^{\12}$. We set
\[ 
G_1=\beta^{-\12} B^*_2 g(K)\left(g'(K)\right)^{-\12}, \ \ G_2=\beta^{\12}\left(g'(K)\right)^{\12} B_2.   
\]
Then 
\[
G_1 G_2 = B^*_2 g(K) B_2, \ \ G_2^* G_2 = \beta B^*_2 g'(K) B_2,
\]
hence in particular $G_1 G_2\leq 0$. 

 By Proposition \ref{prop:fc}, $G_1\in \Psi^{m}(M)$, $G_2\in\Psi^s(M)$, and
\[
\sigma_{\pr}(G_1 G_2)= c g(k), \ \ \sigma_{\pr}(G_2^*G_2)=\beta c  g'(k)
\]
everywhere, modulo lower order terms (using the fact that $c\equiv 0$ in the region where the symbol of $K$ differs from $k$).  Next, recall that $\ham_p  k  >  \beta$ on $\Phi^\pm(L^\pm)$. By Proposition \ref{prop:fc}, the principal symbol of $[P,\i G_1 G_2]-G_2^*G_2$ is
\beq\label{eq:prcom}
 \ham_p c g( k) - \beta c g'( k)= (\ham_p c) g(k)+  c g'(k)(\ham_p k-\beta).
\eeq 
We have $c g'(k)(\ham_p k-\beta)\geq 0$ on the  set $\Phi^\pm(L^\pm)\cup \{c=0\}$, which is a neighborhood of $\Sigma$ thanks to property b) in Lemma \ref{lembetterk}.
Besides, $(\ham_p c) g(k)\geq 0$ on the same set. Thus \eqref{eq:prcom} is $\geq 0$ in a neighborhood of $\Sigma$. 

Since $\wf'(G_2)=\supp c$, applying Lemma \ref{lem:main} yields  
\beq\label{eq:tempa}
\|  G_2 u \| \leq C ( \|B u\|_{{s}}  + \| (P-\omega-\i \epsnu Q) u \|_{{m-s}}  +  \| B_1  u \|_{{s-1/2}} + \ON),
\eeq
for any $B_1,B\in\Psi^0$ such that $\supp c \subset \elll B_1$, and $B$ is elliptic outside a neighborhood of $\Sigma$. In particular we can take $B_1,B$ such that in addition $\wf'(B_1)\subset \{ k<0\}$ and $\wf'(B)\cap \Sigma =\emptyset$. Using the elliptic estimate (Theorem \ref{elliptic}) we can bound $\|A u\|_{s}$ by $\| G_2 u \|$, and bound $\|B u\|_{{s}}$. 

To sum this up, we  have shown for any closed $V\subset\{k<0\}$ that for all $A\in \Psi^0(M)$ with $\wf'(A)\subset V$, there exists $B_1\in\Psi^0$ with $\wf'(B_1)\subset \{ k<0\}$ such that
\beq\label{eq:lem1a}
\|  A u \|_s \leq  C( \| (P-\omega-\i \epsnu Q)u  \|_{{m-s}}  +  \| B_1  u \|_{{s-1/2}} + \ON).
\eeq
\textbf{2.} To show that the $H^{s-\12}(M)$ term in \eqref{eq:lem1a} can be removed, we proceed exactly as in \cite{DZbook}. Namely, by propagation of singularities (Proposition \ref{propagation}) we can estimate $\| B_1  u \|_{{s-1/2}}$ by $\| A  u \|_{{s-1/2}}$ (and other harmless terms), which can be then absorbed into the l.h.s.~using interpolation in Sobolev spaces. \qed      

Next, we obtain a radial estimate which can be interpreted as propagation of regularity into a repulsor from its basin.

\begin{proposition}\label{lem:consp2}  Let $s<-\12$, and let $\delta>0$ be sufficiently small. There exists $B\in\Psi^0(M)$ satisfying $\wf'(B)\cap \Sigma \subset \Phi^+(L^+)\setminus \Lambda^+$, such that  if $A\in\Psi^0(M)$ satisfies $\wf'(A)\cap p^{-1}([-\delta,\delta]) \subset \Phi^+(L^+)$, then for all $N$ and $u\in\cf(M)$,
\beq\label{eq:lem2}
\|  A u \|_{s} \leq C (\|B u\|_{{s}}  +  \|(P-\omega-\i \epsnu Q) u \|_{{s+1}}  + \ON )
\eeq
uniformly in $\epsnu\geq 0$ and $|\omega|\leq \delta$. 
\end{proposition}
\proof We can repeat the proof of Lemma \ref{lem:consp}, with the difference that $K\geq 0$. This entails that $G_1\in\Psi^m(M)$ with $m=1+2s$. Note that \eqref{eq:specialK} is no longer valid, which is why only the case $s<- \12$ is considered.

A further difference is that now  $\ham_p c\geq 0$, and so $(\ham_p c) g(k)\geq 0$ only where $\ham_p c=0$, which  in particular holds true outside a neighborhood of $\{k=0\}$. This has the consequence that when applying Lemma \ref{lem:main} we obtain an analogue of   
\eqref{eq:tempa} with $B$ which is elliptic on larger set, hence the extra  $\|B u\|_{{s}}$ term in \eqref{eq:lem2}. \qed

\subsection{Global estimates}\label{ss:global}

Under an extra non-trapping assumption we can combine Pro\-positions \ref{lem:consp} and \ref{lem:consp2} to get a global estimate (in the same way as radial estimates are combined with propagation of singularities in \cite{DZ} and references therein). Following  \cite{CdV} we make the following definition. 
\begin{definition}\label{def:simple} We say that $P$ has \emph{simple structure} if there exists  a weakly hyperbolic attractor $L^+_0\subset  \p \bSig_0$  and  a weakly hyperbolic repulsor $L^-_0\subset \p \bSig_0$ such that  $\Sigma_0 = \Phi^+(L^+_0)\cup \Phi^-(L^-_0)$ and
\beq\label{eq:ss}
\Phi^+(L^+_0)\setminus \Lambda^+_0 = \Phi^-(L^-_0)\setminus \Lambda^-_0.
\eeq
\end{definition}
 
 It is shown in  \cite{CdV} that the simple structure condition is equivalent to the existence of a global escape function on $\Sigma_0$.  As observed in \cite[Rem.~3.2]{CdV} in that the latter statement is then also valid for neighboring frequencies $\omega\in [-\delta,\delta]$ with $\delta>0$ small enough. Hence  $P-\omega$ has also simple structure for $|\omega|\leq \delta$,  and with the notation of \eqref{eqlpm}, we have
 $$
p^{-1}([-\delta,\delta])= \Phi^+(L^+)\cup \Phi^-(L^-), \quad
 \Phi^+(L^+)\setminus \Lambda^+ = \Phi^-(L^-)\setminus \Lambda^-.
 $$

\begin{proposition}\label{prelap}  Let $s<-\12$. If $P$ has simple structure  then for all $N$ and $u\in\cf(M)$,
\beq\label{eq:lap1}
\|   u \|_{s} \leq  C(  \| (P-\omega-\i \epsnu Q) u \|_{{-s}}   + \ON) 
\eeq
uniformly in $\epsnu\geq 0$ and $|\omega|\leq \delta$ for sufficiently small $\delta$. 
\end{proposition}
\proof  Since by hypothesis $\Sigma = \Phi^+(L^+)\cup \Phi^-(L^-)$, we have $\|   u \|_{s} \leq C ( \| A_+ u\|_s + \| A_- u\|_s)$ for some  $A_\pm\in\Psi^0(M)$ satisfying $\wf'(A_\pm)\cap p^{-1}([-\delta,\delta])\subset\Phi^\pm(L^\pm)$. We first apply Proposition \ref{lem:consp2} with $A=A_+$. This gives an estimate with a $\|B u\|_s$ term on the r.h.s., where $\wf'(B)\cap p^{-1}([-\delta,\delta]) \subset \Phi^-(L^-)$ in view of \eqref{eq:ss}. Finally, we apply Proposition \ref{lem:consp}  twice, once with $A=A_-$, and once with $A=B$, and combine the resulting estimates to get \eqref{eq:lap1}.  \qed

\medskip

We can argue exactly as in \cite{DZ} and get a strict analogue of \cite[Lem.~3.1--3.3]{DZ}. Namely, one obtains that on $[-\delta,\delta]$ there is at most a finite number of eigenvalues, and the eigenvectors are necessarily $C^\infty$. Furthermore, if $0\notin \sp_{\rm pp}(P)$ then for $|\omega|\leq \delta$ and $f\in \cf(M)$, the limit 
\[
u_+\defeq \lim_{\epsnu\to 0+} (P-\omega-\i\epsnu Q)^{-1}f 
\] 
exists in $H^{s}(M)$, $s<-\12$. In addition, $u_+$ is the unique solution to the equation $(P-\omega)u=f$ under the condition $\wf'(u)\subset \Lambda^+$.   Note that the choice of $Q$ plays no role, so we have in particular
\beq\label{eq:corrad}
\lim_{\epsnu\to 0+} (P-\omega-\i\epsnu Q)^{-1}f  = (P-\omega-\i 0)^{-1}f
\eeq
in $H^{-1/2-}(M)$, where $(P-\omega-\i 0)^{-1}f$ is the $\epsnu\to 0+$ limit of $(P-\omega-\i\epsnu)^{-1}f$.
 
 \medskip

In the sequel we will actually use  a variant of Proposition \ref{prelap} with $\epsnu$-dependent bound. 

\begin{proposition}\label{prop:epsi repul}
Let $s<-\12$ and $r\in\open{s,\frac{\ell}{2}}$. If $P$ has simple structure then for all $N$ and $u\in\cf(M)$,
	\beq\label{eq:lapimpr}
	\|   u \|_{r} \leq  C \epsnu^{-(s-r)/(2s-\ell)} (  \| (P-\omega-\i \epsnu Q) u \|_{{-s}}   + \ON) 
	\eeq
 uniformly in $\epsnu\geq 0$ and $|\omega|\leq \delta$ for sufficiently small $\delta>0$. 
\end{proposition}
\proof
 The proof is analogous to Proposition \ref{prelap}, with the difference that we use the  obvious modification of the radial estimates (Propositions \ref{lem:consp}--\ref{lem:consp}) resulting from applying Lemma \ref{lem:main2} instead of Lemma \ref{lem:main}. More precisely, we can show a variant of radial and propagation estimates with the l.h.s.~in the same form as \eqref{var2}, and postpone the argument of interpolation in Sobolev spaces as much as possible (in this way the argument for  removing the $\| B_1  u \|_{{s-1/2}}$ terms apply verbatim).
\qed

\medskip

To lighten the notation a bit we  focus on the case when the viscosity term $Q$ is of order $2$, i.e.~we assume $\ell=2$. By ellipticity and a standard square norm argument,  $(P- \omega-i \epsnu Q)^{-1}$ exists and again by ellipticity,  $(P- \omega-i \epsnu Q)^{-1}\in\Psi^{-2}(M)$. If we assume in addition that $0\notin \sp_{\rm pp}(P)$ then we can get rid of the smoothing error term in the uniform estimates by standard compact embedding arguments. In the sequel we will need the following  version.
 
 \begin{proposition}\label{prop:tes} Assume $\ell=2$, $P$ has simple structure and  $0\notin \sp_{\rm pp}(P)$. Then $(P-\omega - i \epsnu Q)^{-1}=O(\epsnu^{-1/6-})$ in $B(L^2(M),H^{-\12-}(M))$ and $(P-\omega - i \epsnu Q)^{-1}=O(\epsnu^{-1/3})$ in $B(H^{-\12-}(M))$, uniformly in $|\omega|\leq \delta$ for sufficiently small $\delta>0$. 
 \end{proposition}

\begin{proof}
	 By taking $\ell=2$ and $r=0^+$ in \eqref{eq:lapimpr} with $P$ replaced by $-P$ (this merely exchanges attractors and repulsors)  we obtain the uniform estimate 
	\beq\label{eq:radial}
	\|   u \|\leq  C \epsnu^{-1/6-} (  \| (P-\omega+\i \epsnu Q) u \|_{1/2+}   + \ON),
	\eeq
	Let $u_\nu=\nu^{1/6+}(P-\omega+i\nu Q)^{-1}f$ with $f\in L^2(M)$. Then, by taking $N=1/2+$ in \eqref{eq:radial}, we obtain
	\beq \label{eq:apply radial}
	\|u_\nu\|\leq C\|f\|+C\|(P-\omega+i\nu Q)^{-1}f\|_{-1/2-}.
	\eeq
	Next, by the remark following Proposition \ref{prelap}, the second term in the RHS of \eqref{eq:apply radial} is bounded. Thus the family $\epsnu^{1/6+} (P-\omega-i\epsnu Q)^{-1}$ is bounded in the strong operator topology of $B(L^2(M),H^{-\12-}(M))$. By duality, $(P-\omega - i \epsnu Q)^{-1}=O(\epsnu^{-1/6-})$ in $B(L^2(M),H^{-\12-}(M))$. 
\end{proof}

\section{Spectral analysis in the presence of viscosity} \label{s:spectral}

\subsection{Spectrum of $P_\epsnu$}

Recall that $Q\in \Psi^{\ell}(M)$, $\ell\geq 0$, $Q$ is elliptic and $Q>0$. From now on  we assume $\ell=2$.

In the following we denote for all $\epsnu>0$, 
$$
P_\epsnu:=P-i\epsnu Q, \quad \Dom(P_\epsnu) = H^2(M).
$$
 We observe that  $-iP_\epsnu=-\epsnu Q - i P$ is the generator of a strongly continuous one-parameter semigroup of contractions (as it is a bounded perturbation  of $-\epsnu Q$), which we denote  somewhat abusively by  $\left(e^{-itP_\epsnu}\right)_{t\in\rr_+}$. 

More precisely,  by an elementary numerical range argument one gets that
\beq
	\sp(P_\epsnu)\subset  \lbrace \lambda \in\cc  \st \module{\Re \lambda}\leq \|P\|_{B(L^2)},  \ \Im \lambda\leq - \epsnu\rbrace
\eeq
and then for $\lambda\notin \sp(P_\epsnu)$, i.e.~if $\module{\Re \lambda}>\|P\|_{B(L^2)}$ or $\Im \lambda >-\epsnu$
\beq\label{eq:estimeresolvante}
	\|(P_\epsnu-\lambda)^{-1}\|_{B(L^2)}\leq \min\left(  \frac{1}{\module{\module{\Re \lambda}-\|P\|_{B(L^2)}}}, \frac{1}{\module{\Im \lambda + \epsnu}}  \right). 
\eeq
We also note the following Sobolev space bounds on the real line.

\begin{lemma}\label{lem:inv} For all $\epsnu>0$, $(P_\epsnu-\omega)^{-1}\in\Psi^{-2}(M)$. Furthermore:
\beq\label{la1}
\| (P_\epsnu-\omega)^{-1} \|_{B(H^{-1},H^1)}\leq 1/\epsnu
\eeq 
uniformly in $\omega\in\rr$, and 
\beq\label{la2}
\norm{ (P_\epsnu-\omega)^{-1}  \left(P-\omega-i\epsnu\right)}_{B(H^{1},H^{1})}\leq C
\eeq
uniformly in $\epsnu>0$, $\omega\in\rr$.
\end{lemma}
\proof  The operator $P_\epsnu^{-1}$ is the inverse of an elliptic operator in $\Psi^2(M)$ so it belongs to $\Psi^{-2}(M)$.  To prove \eqref{la1} it suffices to observe that 
\[
Q^{\12} (P-\omega-\i \epsnu Q)^{-1} Q^{\12}=(Q^{-\12}(P-\omega) Q^{-\12}-\i \epsnu)^{-1},
\] 
which is the resolvent of a bounded, self-adjoint operator. To see that \eqref{la2} holds true, we write 
\[
(P_\epsnu-\omega)^{-1}  (P-\omega-\i \epsnu)= \one +\i \epsnu (P_\epsnu-\omega)^{-1} (Q-\one) \in B(H^1(M))
\]
where the r.h.s.~is uniformly bounded by \eqref{la1}.
\qed

\begin{lemma} \label{lem:bidon} Suppose $\varphi\in S^0(\rr)$ and  $\lambda\in \rr \setminus \supp \varphi$. Let $\pi_\lambda\in B(H^1(M))$ be the orthogonal projection to $\Ker(P-\lambda)$ in the sense of $H^1(M)$. 
Then 
$$
(P_\epsnu-\lambda)^{-1} \varphi (P) \to (\one -\pi_\lambda)(P-\lambda)^{-1} \varphi (P) 
$$
as $\epsnu \to 0^+$ in the strong operator topology of $H^1(M)$.
\end{lemma}
\begin{proof} Denote $\widetilde{P}_\lambda=Q^{-\12}(P-\lambda)Q^{-\12}$. We remark that $u_\lambda\in L^2(M)$ is in $\Ker \widetilde P_\lambda$ if and only if $Q^{-\12}  u_\lambda \in \Ker (P-\lambda)\cap H^1(M)$. We conclude
\beq\label{eq:pil}
\pi_\lambda=Q^{-\12} \one_{\{0\}}(\widetilde{P}_\lambda) Q^{\12}
\eeq
by comparing the range of both sides and checking self-adjointness of the r.h.s.~in $H^1(M)$ (above,  $\one_{\{0\}}(\widetilde{P}_\lambda)$ is understood in the $L^2(M)$ sense). 
 
By the same computations as in the proof of Lemma \ref{lem:inv} we can write 
$$
(P_\epsnu-\lambda)^{-1}  (P-\lambda)= (\one +\i \epsnu Q^{-\12} (\widetilde{P}_\lambda-i\epsnu)^{-1} Q^{\12}).
$$
The second summand equals
$$
\bea
 Q^{-\12} \i \epsnu (\widetilde{P}_\lambda-i\epsnu)^{-1}  Q^{\12} 
\eea
$$
where $i \epsnu (\widetilde{P}_\lambda-i\epsnu)^{-1}$ tends in the  $B(L^2(M))$ strong operator topology to the spectral projection $-\one_{\{0\}}(\widetilde{P}_\lambda)$   by functional calculus (where $\one_{\{0\}}$ is the characteristic function of $\{0\}$),  hence to $0$ since $\Ker \widetilde{P}_\lambda =\{ 0\}$. In consequence, 
$$
(P_\epsnu-\lambda)^{-1}  (P-\lambda) \to \one - Q^{-\12} \one_{\{0\}}(\widetilde{P}_\lambda) Q^{\12}  = \one - \pi_\lambda
$$
strongly as operators in $B(H^1(M))$. Furthermore $B=(P-\lambda)^{-1}\varphi(P)\in \Psi^0(M)$ by Proposition \ref{prop:zero}. We conclude
$$
\bea
(P_\epsnu-\lambda)^{-1}\varphi(P)&=\big((P_\epsnu-\lambda)^{-1}(P-\lambda)) (P-\lambda)^{-1}\varphi(P)  \fantom\to (1-\pi_\lambda)(P-\lambda)^{-1}\varphi(P)
\eea
$$ 
strongly.  
\end{proof}

\subsection{Spectral representation of the semi-group}

In the sequel we  will use the following contour integral representation: 
\begin{equation}\label{eq:srs}
	e^{-iP_\epsnu t}=-(2\pi i)^{-1}\int_\Gamma (P_\epsnu-z)^{-1}e^{-izt}\mathrm{d}z,\quad t >0,
\end{equation}
where
\begin{equation}\label{eq:Gamma}
\bea
\Gamma:=\Gamma_0\cup\Gamma_-\cup\Gamma_+&:=\left[-\|P\|_{B(L^2)}-\delta,\|P\|_{B(L^2)}+\delta\right]\\
	& \phantom{:=} \ \cup\left\lbrace-\|P\|_{B(L^2)}-\delta-r e^{i\beta},r\in\clopen{0,\infty}\right\rbrace\\
	&\phantom{:=} \ \cup\left\lbrace\|P\|_{B(L^2)}+\delta+re^{-i\beta},r\in\clopen{0,\infty}\right\rbrace
\eea\end{equation}
with $\delta>0$ small  and $\beta\in \open{  0,\pi/2}$. Note that $\Gamma$ encloses $\sp(P_\epsnu)$. As $z\mapsto (P_\epsnu-z)^{-1}$ is well-defined and bounded on $\Gamma_\epsnu$ and $|e^{-izt}|\leq |e^{(\Im z) t}|$ and $\Im z<0$ on $\Gamma_\pm$ so the integral is well-defined sense for all $t>0$. The formula can be shown easily e.g.~by an argument analogous to \cite[Thm.~1.7.7]{Pazy}.

\section{Multiscale analysis of the solution to the forced equation}

\label{s:final} 

\subsection{Proof of main result} Recall that we want to study the $t\to +\infty$ behaviour of the solution of the initial value problem  
\beq\label{eq:IVP}
\begin{cases}
		i\partial _tu_\epsnu-(P-i\epsnu Q) u_\epsnu=f,\\
		u_\epsnu(0)=0,
\end{cases}
\eeq
with  forcing $f\in C^\infty(M)$ in the low viscosity regime $\epsnu\to 0+$. Note that if we change  $P$ to $P-\omega$ this amounts to merely changing  the forcing term $f$ to $e^{-i\omega t}f$. By Duhamel formula, we have  
\beq\label{eq:ut}
	u_\epsnu(t)={-i}\int_0^t e^{-isP_\epsnu}f\mathrm{d}s  = P_{\epsnu}^{-1} (e^{-it P_\epsnu}-1) f.
\eeq
Note that it is relatively straightforward to show using the last formula in \eqref{eq:ut} that for each $\epsnu>0$, $\| u_\epsnu(t)\|$ is bounded, but the dependence on $\epsnu$ is pretty bad, namely
$$
\| u_\epsnu(t)\| \leq\epsnu^{-1} C \|f \|, 
$$
where the $\epsnu^{-1}$ factor comes from the estimate $\| P_{\epsnu}^{-1}\|_{B(L^2)} \leq \epsnu^{-1}$ . 

 Moreover, we have some rough results on the convergence of the solution if we fix $\nu$ or $t$.  Namely, using \cite[Thm.~3.30]{Davies1981} we can show that 
	\beq
		\|e^{-itP\epsnu}f\|\leq e^{-t\epsnu}\|f\|,
	\eeq
	and this shows that for any fixed $\epsnu>0$, \beq 
	\lim_{t\rightarrow \infty}u_{\epsnu}(t)=-P_{\epsnu}^{-1}f. 
	\eeq
	On the other hand, by combining \cite[Thm.~4.2]{Pazy} and Lebesgue's theorem in formula \eqref{eq:ut}, we get that for all $t$ in any compact interval 
	\beq 
		\lim_{\epsnu\rightarrow 0^+}u_\nu(t)=u_0(t),
	\eeq
	where $u_0$ is the solution of \eqref{eq:sve}. The more difficult question is however to usefully combine both limits in a suitable regime for $t$ and $\nu$. 

\begin{theorem} \label{thm:main}
	Assume that $P$  has simple  structure and $0\notin \sp_{\rm pp}(P)$. Then for any $f\in C^\infty(M)$, the solution  of \eqref{eq:IVP} decomposes as  \beq\label{eq:decomp}
	u_\epsnu(t)=u_{\epsnu,\infty}+b_\epsnu(t)+e_{\epsnu}(t),	
	\eeq
	where $u_{\epsnu,\infty}= -P_\epsnu^{-1} f$ {converges to $-(P-i0^+)^{-1}f$ in} $H^{-\frac{1}{2}-}(M)$, $ \| b_\epsnu(t)\|\leq C\|f\|_1$ 	uniformly in  $t> 0$, $\epsnu> 0$,  and for all $\delta_1>0$ there exists $\delta_2>0$ such that
	\beq\label{eq:decomp2}
	 \quad \|e_{\epsnu}(t)\|_{-1/2-}\leq  C t^{-\delta_2}  \| f\|, 
	\eeq
		uniformly for  $t \sim \nu^{-\frac{1}{3}-\delta_1}$.  
\end{theorem}
\proof
We use the integral representation \eqref{eq:srs} of the semigroup, namely,
\beq
e^{-isP_\epsnu}f=\textcolor{blue}{-}\frac{1}{2\pi i}\int_\Gamma (P_\epsnu-z)^{-1}e^{-izs} f \mathrm{d}z.\label{eq:int to estime}
\eeq
Next, we split the integral over $\Gamma$ into the sum of three integrals over $\Gamma_0$, $\Gamma_+$ and $\Gamma_-$ (where the different $\Gamma_\#$ are defined in \eqref{eq:Gamma}) which we denote respectively by $I_0(s)$, $I_+(s)$ and $I_-(s)$.  We can assume without loss of generality that there are no  isolated eigenvalues of $P$ which are not accumulation points of $\bigcup_{\nu>0}\sp_{\rm pp}(P_\nu)$, otherwise we  can slightly deform  $\Gamma_0$ to bypass these eigenvalues.

Let $\chi\in \cf_{\rm c}(\rr;[0,1])$ be such that $\chi\equiv 1$ in a neighborhood of $0$ and $\chi\equiv 0$ on $\rr\setminus\open{-\delta,\delta}$. Let $\varphi\in \cf_{\rm c}(\rr;[0,1])$ has the same properties and in addition $\supp\varphi\Subset{\chi^{-1}(1)}$.  We further split the $I_0(s)$ integral into two terms: 
$$
\bea
I_0(s)=:  I_{0,\chi}(s) + I_{0,1-\chi}(s), \quad   I_{0,\chi}(s) :=  \textcolor{blue}{-}\frac{1}{2\pi i}  \int_{\Gamma_0} \chi(\lambda)  (P_\epsnu-\lambda)^{-1}e^{-i\lambda s} \mathrm{d}\lambda.
\eea
$$
For the decomposition \eqref{eq:decomp} of $u_\epsnu(t)$ we take 
\beq\bea\label{eqdefcom}
b_\epsnu(t)&:=e^{-it P_\epsnu}P_\epsnu^{-1}(\one-\varphi(P))f+{i}\int_t^\infty (I_++I_- + I_{0,1-\chi})(s)\varphi(P)f {\rm d} s, \\ e_\epsnu(t)&:= {i}\int_t^\infty I_{0,\chi}(s)\varphi(P)f {\rm d} s.
\eea\eeq
We will show that the integrals converge, in which case for every $\epsnu >0$ we have
$$
{i}\int_t^\infty (I_++I_- + I_0)(s) {\rm d}s = {i}\int_t^\infty e^{-i t P_\epsnu} {\rm d}s = P_{\epsnu}^{-1}e^{-i s P_\epsnu}.
$$
Thus, $b_\epsnu(t)+e_\epsnu(t)= e^{-it P_\epsnu}P_\epsnu^{-1} f$, so $u_\epsnu(t)=u_{\epsnu,\infty}+ b_\epsnu(t)+e_\epsnu(t)$ by \eqref{eq:ut} indeed.

We start by estimating the first summand in  formula \eqref{eqdefcom} for $b_\epsnu(t)$. We write
$$
\bea
\| e^{-it P_\epsnu}P_\epsnu^{-1}(\one-\varphi(P)) f\| &\leq   \| P_\epsnu^{-1}(\one-\varphi(P)) f\|  \\
& \leq  C \| P_\epsnu^{-1}(\one-\varphi(P)) f\|_1  \\ 
&\leq C\|f\|_{1},
\eea
$$
where by Lemma \ref{lem:bidon}, $\| P_\epsnu^{-1}(1-\varphi(P))\|_{B(H^1)}$ is uniformly bounded because of $P_\epsnu^{-1}(1-\varphi(P))$ strongly converging in $H^1(M)$ and by the uniform boundedness principle. 

Next, the terms in \eqref{eqdefcom}  involving $I_\pm(s) f$ are easily bounded by $\| f\|$ since on the contour {\eqref{eq:estimeresolvante} holds} and the $e^{-i\lambda s}$ factor gives exponential decay along the contour. To estimate the term involving $I_{0,1-\chi}(s)$ we study the limit  $\epsnu\to 0^+$  and integrate by parts  
\beq \label{eq:ipp}
\bea
& \int_{\Gamma_0} (1-\chi)(\lambda)  (P_\epsnu-\lambda)^{-1} \varphi(P)  e^{-i\lambda s}  f \mathrm{d}\lambda  \\
& \to \int_{\Gamma_0} (1-\chi)(\lambda)  (1-\pi_\lambda)(P-\lambda)^{-1} \varphi(P)  e^{-i\lambda s}  f \mathrm{d}\lambda  \\
&=  \int_{\Gamma_0} (1-\chi)(\lambda)  (P-\lambda)^{-1} \varphi(P)  e^{-i\lambda s}  f \mathrm{d}\lambda  \\
&= s^{-2}\int_{\Gamma_0} \frac{d^2}{d\lambda^2}\left((1-\chi)(\lambda)  (P-\lambda)^{-1}\right) \varphi(P)  e^{-i\lambda s}  f \mathrm{d}\lambda.
\eea 
\eeq
Above, the convergence as $\epsnu\to 0^+$ comes from the fact that for  $\lambda\in \supp({1-\chi})\cap \Gamma_0$ we can use Lemma \ref{lem:bidon},  and then to go from the second line to the third we notice that as a function of $\lambda$, $\pi_\lambda$ is supported on a set of Lebesgue measure $0$ since $H^1(M)$ eigenvalues of $P$ are a countable set. By integrating  the resulting estimate in $s$ we get
$$
\norm{\int_t^\infty I_{0,\chi}(s)\varphi(P)f {\rm d} s}_1\leq C \| f\|_1.
$$
In conclusion, the estimates obtained so far give $\| b_\epsnu(t)\|\leq C \| f\|_1$.

We now estimate $e_\epsnu(t)$, which is obtained by integrating in $s$ the expression
$$
 I_{0,\chi}(s)\varphi(P)f = - \frac{1}{2\pi i}  \int_{\Gamma_0} \chi(\lambda)  (P_\epsnu-\lambda)^{-1}e^{-i\lambda s} \varphi(P)f \mathrm{d}\lambda.
$$
 Since $0\notin \sp_{\rm pp}(P)$ we can apply Proposition \ref{prop:tes} which says that that for all $f\in {C}^\infty(M)$,
\begin{equation}\label{eq:H1/2 norm}
	\|(P_\epsnu-\lambda)^{-1}f\|_{{-1/2-}}\leq C \epsnu^{-1/6}\|f\|
\end{equation}
and 
\begin{equation}\label{eq:H1/2 norm2}
	\|(P_\epsnu-\lambda)^{-1}f\|_{{-1/2-}}\leq C \epsnu^{-1/3}\|f\|_{-1/2-}
\end{equation}
uniformly in  $\lambda\in \supp\chi\subset [-\delta,\delta]$.  By integrating by parts $n\geq 1$ times we obtain 
$$
I_{0,\chi}(s)\varphi(P)f =  -\frac{1}{2\pi i}  i^n s^{-n} \int_{\Gamma_0} \frac{d^n}{d\lambda^n}\left( \chi(\lambda)   (P_\epsnu-\lambda)^{-1}\right)e^{-i\lambda s} \varphi(P)f \mathrm{d}\lambda.
$$
To bound  the $H^{-1/2-}(M)$ norm of  $I_0(s) f$ we use the Leibniz rule and then   estimate $(P_\epsnu-\lambda)^{-k} f$ for $k\leq n$ and $\lambda\in \supp\chi$. To that end we use \eqref{eq:H1/2 norm} once and  \eqref{eq:H1/2 norm2} at most $n-1$ times. This gives
\beq\label{ide3}
\norm{ I_{0,\chi}(s)\varphi(P) f}_{-1/2-}\leq C\|f\|  s^{-n}\epsnu^{{1/6-n/3}}.
\eeq
Integrating the estimate \eqref{ide3} yields
\beq\label{eq:eest}
\| e_\epsnu(t)\|_{-1/2-}\leq   C t^{-n+1}\epsnu^{{1/6-n/3}} \|f\| .
\eeq
Since $n$ can be taken arbitrarily large we conclude the bound on  $\|e_{\epsnu}(t)\|_{-1/2-}$.
\qed

\medskip

\begin{proposition}\label{prop:main} With the same assumptions and notation as in Theorem \ref{thm:main}, if in addition $f\in \Ran \one_{[-\delta,\delta]}(P)$ for $\delta>0$ small enough, then for each $\alpha>0$ 
$$
\lim_{\nu\rightarrow 0}\sup_{t\in \opens{\nu^{\frac{-1-\alpha}{3}},\infty}}\|u_{\nu}(t)-u_{0,\infty}(t)\|_{-1/2-}=0.
$$	
\end{proposition}
\begin{proof}
As $\lim_{\nu\rightarrow 0^+}u_{\nu,\infty}=u_{0,\infty}$ in $H^{-1/2-}(M)$, it suffices to prove that $e_\nu(t)$ and $b_\nu(t)$ converge to $0$ in the requested regime. This follows by inspection of the proof of Theorem \ref{thm:main}. More precisely, the  claim for  $e_\nu(t)$ follows from \eqref{eq:eest} therein.  For $\delta$ sufficiently small, $b_\nu(t)$ simplifies to
$$
b_\epsnu(t)={i}\int_t^\infty (I_++I_- + I_{0,1-\chi})(s)\varphi(P)f {\rm d} s.
$$
The terms involving $I_+$ and $I_-$ are easy to handle because their integrants can be bounded exponentially and uniformly with respect to $\nu$ thanks to \eqref{eq:estimeresolvante}. Finally the  $I_{0,1-\chi}$ term is dealt with by noticing that the argument in \eqref{eq:ipp} gives  as much decay as wanted.
\end{proof}

\medskip

\begin{remark} If the assumption  $f\in \Ran \one_{[-\delta,\delta]}(P)$ is dropped then $b_\nu(t)$ involves an extra term $e^{-it P_\epsnu}P_\epsnu^{-1}(\one-\varphi(P))f$ which is not known to decay. With the methods in this paper we could represent it by a contour integral and try to use arguments similar to the way we treat other terms, but getting the desired decay rate would require resolvent estimates only known to hold in neighborhood of $0$ of the spectrum with the present assumptions---away from $0$ we de not make any dynamical assumption so only $\nu^{-1}$ estimates are available.
\end{remark}

\appendix
\section{Preliminaries on pseudo-differential calculus} \label{app:A}

\subsection{Basic estimates}\label{ss:basic} Let us recall the following well-known \emph{elliptic estimate}, see e.g.~\cite[Thm.~E.32]{DZbook}.

\begin{theorem}\label{elliptic} Let $A_1\in\Psi^0(M)$, $A_2\in\Psi^\ell(M)$, $\ell\in\rr$. Assume that $\wf(A_1)\subset \elll(A_2)$. Let $s,N\in\rr$. Then for each $u\in\cD'(M)$, if $A_2 u\in H^{s-\ell}(M)$, then $A_1 u \in H^s(M)$ and
\[
\| A_1 u\|_{s}\leq C ( \| A_2 u \|_{{s-\ell}} + \| u\|_{{-N}}).
\]
\end{theorem} 

In our context, with $P$ and $Q$ as in Section \ref{s:radial}, another useful version of the elliptic estimate that follows from the same proof is the following statement: if $\wf(A_1)\subset \elll(A_2)$, $\wf(A_1)\cap p^{-1}([-\delta,\delta])=\emptyset$, and if $A_2 (P-\omega+ i \epsnu Q) u \in H^{s-\ell}(M)$, then $A_1 u \in H^s(M)$ and
\beq
\| A_1 u\|_{s}\leq C (\| A_2 (P-\omega- i \epsnu Q) u \|_{{s-\ell}} +  \| u\|_{{-N}})
\eeq
 uniformly in $\epsnu>0$ and $|\omega|\leq \delta$.

The proposition below is a microlocal version of the \emph{sharp G\r{a}rding inequality}, see e.g.~\cite[Prop.~E.23]{DZbook} for the proof.

\begin{proposition}\label{garding} Let $A\in\Psi^{2s}(M)$, $B\in\Psi^0(M)$ $B_1\in\Psi^0(M)$, $s\in\rr$. Suppose 
\[
\sigma_{\pr}(A)\geq 0 \mbox{ on } T^*M\setminus\elll(B),
\]
and $\wf(A)\subset \elll(B_1)$. Then for each $N$ and all $u\in\cf(M)$,
\[\bra Au,u\ket \geq - C( \| B u \|^{2}_{{s}}+  \| B_1 u \|^{2}_{{s-1/2}} + \| u \|^2_{{-N}}).
\]
\end{proposition}

\subsection{Functions of pseudo-differential operators} The next proposition allows to compute the principal symbol of functions of pseudo-differential operators, defined using the functional calculus for self-adjoint operators.  

\begin{proposition}\label{prop:fc}  Let $m\geq 0$. Assume $A\in \Psi^{m}(M)$ is elliptic and self-adjoint in the sense of operators on $L^2(M)$. Let $g\in S^{p}(\rr)$, $p\in \rr$. Then $g(A)\in\Psi^{mp}(M)$ and $S^p(\rr)\ni g\mapsto g(A)\in \Psi^m(M)$ is continuous. Moreover, if $g$ is elliptic in $S^{p}(\rr)$  then
\beq 
\sigma_{\pr}( g(A) )= g_{\pr}(\sigma_{\pr}(A)),
\eeq
where $g_{\pr}$ is the principal symbol of $g$. 
\end{proposition}
\proof This follows from well-known arguments, see e.g.~\cite[Thm.~5.4]{Ro}, \cite[Corr.~4.5]{Bo}, \cite[Prop.~4.2]{GW0} for the $\rr^d$ case; cf.~\cite{DS} for the semi-classical case.  The standard proof proceeds by applying Beals' criterion to $g(A)$. By  Helffer--Sjöstrand formula this then reduces to  applying Beals' criterion to the resolvent $(A-\lambda)^{-1}$, which is straightforward using ellipticity. In our setting the only necessary adaptation is the use of a variant of Beals' criterion on compact manifolds, see e.g.~\cite[\S5.3]{Ruzhansky2010}.  The continuity statement follows from the fact that when using Beal's criterion, seminorms in $\Psi^0(M)$ are estimated through norms of iterated commutators of vector fields with $g(A)$. Again, this boils down to controlling iterated commutators of vector fields with the $(A-\lambda)^{-1}$, so the dependence on $g$ only arises through  integration  with an almost analytic extension, which depends continuously on $g\in S^p(\rr)$.     \qed

\medskip

In the case when the order is $m=0$ the ellipticity assumption can be removed.

\begin{proposition}\label{prop:zero} Let $P\in \Psi^0(M)$ and $P^*=P$. Let $f\in S^{p}(\rr)$, $p\in \rr$. Then $f(P)\in\Psi^{0}(M)$. Futhermore, $S^{p}(\rr)\ni f\mapsto f(P)\in \Psi^0(M)$ is continuous.
\end{proposition}
\proof Since $p$ is bounded, the operator $P-\omega$ is elliptic for sufficiently large $\omega\geq 0$. We let $g(\lambda)=f(\lambda+\omega)$ and apply Proposition \ref{prop:fc} to $g(P-\omega)=f(P)$.  
\qed

{\small
\subsubsection*{Acknowledgments} 
The authors would like to thank particularly Yves Colin de Verdière and Jian Wang   for  stimulating discussions and helpful suggestions. Support from the grant ANR-20-CE40-0018 of the Agence Nationale de la Recherche is gratefully acknowledged.  N.F.~gratefully acknowledges support from the Région Pays de la Loire via the Connect Talent Project HiFrAn 2022, 07750.    \medskip }

 \bibliographystyle{abbrv}
 \bibliography{viscosity}

\end{document}